\documentclass[journal]{IEEEtran}

\ifCLASSINFOpdf
   \usepackage[pdftex]{graphicx}
   \graphicspath{{../pdf/}{../jpeg/}}
   \DeclareGraphicsExtensions{.pdf,.jpeg,.png}
\else
   \usepackage[dvips]{graphicx}
  \graphicspath{{../eps/}}
  \DeclareGraphicsExtensions{.eps}
\fi

\usepackage{amssymb}
\usepackage{amsmath}
\usepackage{amsmath, bm}
\usepackage{amsfonts}
\usepackage{amsthm}
\usepackage{url}

\usepackage{algorithm}
\usepackage{algorithmic} 
\usepackage{enumerate}

\theoremstyle{definition}

\newtheorem{theorem}{Theorem}
\newtheorem{definition}{Definition}
\theoremstyle{remark}

\usepackage{multirow}
\usepackage{booktabs}
\usepackage{makecell}
\usepackage{threeparttable}
\usepackage{array}
\usepackage{footnote}

\usepackage{xcolor}

\begin{document}

\title{Non-Iterative Solution for Coordinated Optimal Dispatch via Equivalent Projection---Part I: Theory}






\author{Zhenfei~Tan,~\IEEEmembership{Member,~IEEE,}
      Zheng~Yan,~\IEEEmembership{Senior~Member,~IEEE,}
      Haiwang Zhong,~\IEEEmembership{Senior~Member,~IEEE,}
      and~Qing~Xia,~\IEEEmembership{Senior~Member,~IEEE}
      \vspace{-2em}
\thanks{Z. Tan and Z. Yan are with the Key Laboratory of Control of Power Transmission and Conversion (Shanghai Jiao Tong University), Ministry of Education, Shanghai 200240, China. H. Zhong and Q. Xia are with the State Key Laboratory of Power Systems, the Department of Electrical Engineering, Tsinghua University, Beijing 100084, China. Corresponding Author: Z. Yan (e-mail: yanz@sjtu.edu.cn).}

}



%



\maketitle

\begin{abstract}
  Coordinated optimal dispatch is of utmost importance for the efficient and secure operation of hierarchically structured power systems. Conventional coordinated optimization methods, such as the Lagrangian relaxation and Benders decomposition, require iterative information exchange among subsystems. Iterative coordination methods have drawbacks including slow convergence, risk of oscillation and divergence, and incapability of multi-level optimization problems. To this end, this paper aims at the non-iterative coordinated optimization method for hierarchical power systems. The theory of the equivalent projection (EP) is proposed, which makes external equivalence of the optimal dispatch model of the subsystem. Based on the EP theory, a coordinated optimization framework is developed, where each subsystem submits the EP model as a substitute for its original model to participate in the cross-system coordination. The proposed coordination framework is proven to guarantee the same optimality as the joint optimization, with additional benefits of avoiding iterative information exchange, protecting privacy, compatibility with practical dispatch scheme, and capability of multi-level problems. 
\end{abstract}


\begin{IEEEkeywords}
  Coordinated optimization, equivalent model, optimal dispatch, projection, non-iterative.
\end{IEEEkeywords}

%
\IEEEpeerreviewmaketitle

\bstctlcite{IEEEexample:BSTcontrol} 
\section{Introduction}\label{sec:intro}
\subsection{Background and Motivation}
\IEEEPARstart{M}{odern} electric power systems are structured and managed hierarchically, stretching from the inter-regional network, to the regional transmission system, high/medium/low-voltage distribution system, and user-side microgrid. With the rapid development of ultra-high-voltage transmissions, the surge of distributed flexible resources, and the integration of multiple energy sectors, systems in different regions and levels are coupled more closely. The coordinated optimal dispatch (COD) among different subsystems is becoming increasingly important to facilitate the economic and secure operation of the entire system.

Owing to technical restrictions (such as limited computation capacity) and regulatory reasons (such as management independence and information privacy), it is impractical to dispatch the hierarchical system by the centralized optimization that requires the collection of global information \cite{lit:dual_6}. In this regard, researchers introduce coordinated optimization methods to realize the COD in a decoupled fashion. The coordinated optimization decomposes the joint optimization problem into an upper-level problem and a series of lower-level problems. These problems are solved by corresponding subsystem operators independently and the overall optimum is achieved via the iterative information exchange among the upper-and lower-level systems. A broad spectrum of coordinated optimization methods has been developed with applications to various scenarios, e.g., the multi-area coordinated dispatch, transmission-distribution coordinated dispatch, and the coordinated control of distributed resources. However, conventional coordinated optimization methods rely on iterative information interaction, which has drawbacks including, 1) convergence issue: the iteration process may converge slowly, oscillate, or even diverge under some system operation states, and the convergence performance of some methods also depends on the parameter tuning; 2) communication burden: the iterative information exchange will occupy the communication channel for a long time until the iteration convergences; 3) scalability issue: the iteration number increases sharply with the number of system levels and subsystems \cite{ref:nested}, which impedes the coordination of hierarchical systems with multiple levels and multiple subsystems; and 4) compatibility issue: subsystems in different regions and levels are coordinated in a single-round serial manner in the real-world dispatch, which repels the iterative coordination method that requires repeated information exchange \cite{ref:iteration}. These drawbacks restrict the practical application of iterative coordination methods. Realizing the coordinated optimization in a non-iterative fashion will radically overcome the drawbacks above, which is a longstanding expectation from the industry and academia.

\begin{table*}[t]
  \caption{Review of Coordinated Optimal Dispatch Methods}
  \label{tab:review_dco}
  \centering
  \begin{threeparttable}
  \setlength{\tabcolsep}{3.5mm}
  \begin{tabular}{lllllll}
  \toprule
  Method & Decomposed & \makecell[l]{No \\iteration} & \makecell[l]{Privacy \\protection} &  \makecell[l]{Multi-level \\ coordination} & \makecell[l]{Communication \\burden} & \makecell[l]{Computation\\ burden \tnote{1}} \\ 
  \midrule
  Centralized optimization & $\times$ & \checkmark & $\times$  & \checkmark & Low & High\\
  Encrypted optimization [26], [27] & $\times$ & \checkmark & \checkmark  & \checkmark & Low & High \\
  Primal decomposition & & & & & & \\
  - Benders decomposition [5], [6] & \checkmark & $\times$ & \checkmark  & $\times$ & High & Medium \\
  - Multi-parametric programming [7]  & \checkmark & $\times$ & \checkmark  & $\times$ & High & Medium \\
  - Marginal equivalence [8], [17] & \checkmark & $\times$ & \checkmark & $\times$ & High & Medium\\
  - \textbf{Equivalent projection (this paper)} & \checkmark & \checkmark & \checkmark  & \checkmark & Low & Medium\\
  Dual decomposition [1], [9-13], [19]  & \checkmark & $\times$ & \checkmark & $\times$ & High & Medium \\
  Consensus algorithm [14-16] & \checkmark & $\times$ & \checkmark & $\times$ & High & Low \\ 
  \bottomrule
  \end{tabular}
  \begin{tablenotes}
    \item[1] Maximum computation resource requirement of upper-level and lower-level subsystems.
  \end{tablenotes}
  \end{threeparttable}
  \end{table*}

To this end, this work aims at the non-iterative coordinated optimization method, i.e., coordinating the dispatch of multiple subsystems without iterative information exchange. The cause that conventional coordinated optimization requires iterations is: the information exchanged in each round of iteration does not represent complete technical and economic features of subsystems and thus, making the upper-level decision conflict with the feasibility or optimality of subsystems. Consequently, iterative information exchange is needed to revise the upper-level decision. To address this issue, this study proposes a novel system reduction theory termed the equivalent projection (EP), which makes external equivalence of the subsystem with much fewer data. The EP model depicts complete technical and economic information of the subsystem using a group of inequalities regarding coordination variables, and can be used as a substitute for the original subsystem model in the COD. The EP model ensures that the upper-level decision is consistent with the feasibility and optimality of the lower-level system, which avoids iterative information exchange among subsystems and realizes the COD in an iteration-free fashion. 

The EP theory and the non-iterative COD framework are introduced in Part I of this paper. To calculate the EP, a novel polyhedral projection algorithm is developed in Part II, along with detailed applications of the proposed COD method. 


\subsection{Literature Review}
Existing studies on the coordinated optimization and system reduction are reviewed, as summarized in TABLE \ref{tab:review_dco} and TABLE \ref{tab:review_reduction}, respectively.

\subsubsection{Coordinated dispatch}
Applications of the coordinated optimization in power systems date back to the 1980s \cite{ref:dco_first}. Existing coordinated optimization methods can be generally classified into two categories, i.e., the primal decomposition and the dual decomposition. The former decomposes the joint optimization problem by splitting the coupling variables and coordinates the solution of the upper and lower problems by updating coupling variables. Typical primal decomposition methods include: the Benders decomposition \cite{ref:primal_1}, generalized Benders decomposition \cite{ref:primal_2}, multi-parametric programming method \cite{ref:primal_3}, marginal equivalence method \cite{ref:primal_5}, etc. The latter decomposes the joint optimization by relaxing the coupling constraints to the objective function, and coordinates the solution of upper and lower problems by updating dual multipliers of coupling constraints. Typical relaxation methods include: the Lagrangian relaxation \cite{lit:dual_1}, augmented Lagrangian relaxation \cite{lit:dual_2}, optimality condition decomposition \cite{lit:dual_3}, etc. These relaxation methods can be combined with different multiplier updating algorithms, e.g., sub-gradient method, alternating direction method of multipliers (ADMM) \cite{lit:dual_4}, cutting plan method \cite{lit:dual_5}, and dynamic multiplier updating method \cite{lit:dual_6}. In addition to the coordinated optimization, fully distributed dispatch methods are also developed based on the consensus algorithm \cite{lit:consensus, lit:review3}. These methods only require algebraic calculation for each subsystem and thus have very low computation expense. However, the communication burden will increase as much more iterations are needed by these methods.

The coordinated optimization realizes the decomposed optimal dispatch of interconnected subsystems by merely exchanging some boundary information of each subsystem, which protects the information privacy and alleviating the computation burden of the centralized optimization. In the literature, applications of the coordinated optimization are studied for different problems, e.g., multi-regional economic dispatch \cite{ref:app_1}, market integration of multiple regions \cite{ref:app_2}, transmission-distribution system coordination \cite{ref:app_3}, and coordinated dispatch of multi-energy systems \cite{ref:app_4}. In addition to the optimal dispatch, coordinated optimization methods are also applied in areas of the optimal planning \cite{ref:app_6}, coordinated state estimation \cite{ref:app_5}, and optimal voltage/Var control \cite{ref:app_7}. Reviews of coordinated dispatch methods in power system can be found in \cite{ref:iteration, lit:dco_review2}.

However, the requirement of iterative information exchange brings computational and practical drawbacks to existing coordinated optimization methods, as analyzed in the former section. Some recent references have recognized the significance of the non-iterative solution for coordinated optimization. In reference \cite{ref:dco_non_iterative}, the non-iterative transmission-distribution system coordination is realized based on the aggregated cost function of the distribution system. In this study, however, the dimension of coordination variable is restricted to 1, which greatly simplifies the problem setup. In \cite{ref:tan4}, a variable and constraint elimination method is proposed to realize the non-iterative COD of the multi-area power system, in which high-dimensional coordination variables are incorporated. Nevertheless, the applicability to multi-level systems is not addressed in this work. References \cite{ref:privacy1} and \cite{ref:privacy2} solve the coordinated optimization in a centralized manner and protect information privacy of subsystems through data encryption. Iterations are avoided in these studies, but computational barriers exist since the coordinator has to solve a large-scale joint optimization problem. Furthermore, the coordinated optimization results can only be decoded by each subsystem, which may hinder the transparency of dispatch and pricing. Basic theory and method for the non-iterative, privacy-protected, and computationally efficient COD methods remain challenging issues, which motivates the present work.

\subsubsection{System reduction}

\begin{table}
  \caption{Review of System Reduction Methods}
  \label{tab:review_reduction}
  \centering
  \begin{tabular}{llll}
    \toprule
    \multirow{2}{*}{Method} & \multicolumn{3}{c}{Considered elements}\\
    \cline{2-4}
     & \makecell[l]{Equality \\constraints} & \makecell[l]{Inequality \\constraints} & \makecell[l]{Objective \\function} \\
    \midrule
    Network reduction [28], [29] & \checkmark & $\times$ & $\times$\\
    Equivalent line limit [30] & \checkmark & \checkmark & $\times$ \\
    \makecell[l]{Feasible region projection\\and aggregation [31-36]} & \checkmark & \checkmark & $\times$ \\
    \textbf{\makecell[l]{Equivalent projection \\(this paper)}} & \checkmark & \checkmark & \checkmark \\
    \bottomrule
  \end{tabular}
\end{table}

To simplify the analysis and calculation of large-scale interconnected power systems, it is desired to reduce the scale (number of variables and/or constraints) of the system model \cite{ref:net_eq}. The most common reduction technique is the network equivalence, e.g., the Ward equivalence \cite{ref:ward}, which eliminates internal variables from the network equation of an electric grid and yields the equivalent network model at boundary nodes. The network reduction is widely applied in power flow calculation, security analysis, state estimation, and other static analysis problems. However, the network reduction is incapable of coordinated optimization problems since operation limits and cost functions of resources are not incorporated in the reduction. 

To address this issue, an equivalent method preserving transmission limits is proposed in \cite{ref:net_reduce_cst}, which captures the thermal limits of equivalent lines. However, the network is assumed unloaded in this work, and resource operation limits are omitted in the equivalent model. In the authors' previous work \cite{ref:tan1}, projection methods are proposed to reduce the operation constraints of the regional transmission system to the tie-line boundary, which can be used to enforce internal constraints of each area in the cross-area power trading. In \cite{ref:minkowski1}, the Minkowski addition is used to estimate and aggregate the flexibility of distributed resources. In \cite{ref:tan2} and \cite{ref:tan3}, the projection-based reduction method is studied to capture the allowable range of active and reactive power output of virtual power plants. These studies focus on enforcing technical constraints in the reduced model, but do not consider the operation cost function of the system. Reference \cite{ref:minkowski} proposes a Zonotope-based method for characterizing the flexibility region of distributed resources with explicit pricing of each Zonotope parameter. Reference \cite{ref:dist_pq} characterizes flexibility regions of the distribution network given different operation cost levels. Since cost functions are modeled separately from the projection of operation constraints, the accuracy of these methods is limited. Additionally, existing methods are developed for specific application scenarios. There is a research gap in general system reduction theory incorporating both operation constraints and objective function of the system, which will be addressed in this work.


\subsection{Contributions and Paper Organization}
The contribution of Part I is twofold,

1) A novel system reduction theory namely the EP is proposed, which makes external equivalence of the entire optimization model with much fewer variables and constraints. The EP model is proven to capture identical technical and economic operation characteristics of a subsystem without revealing private information.

2) A novel coordinated optimization framework is developed based on the EP, which supports the COD of multi-level systems in a decomposed, iteration-free, and privacy-protected fashion. This coordination framework also has the computational advantage by reducing the scale of the subsystem model.

The paper is organized as follows. Section II introduces the joint optimal dispatch problem and its primal decomposition. Section III introduces the EP theory. Section IV develops the EP-based coordinated optimization method and discusses its properties. Section V illustrates the proposed coordination method with a small example. Section VI concludes this paper.

\section{Problem Formulation}\label{sec:formulation}
\subsection{Basic Notation}
Consider the COD of a system composed of multiple levels on a short-term basis (e.g., day-ahead, hour-ahead, and 5 minute-ahead). Assume the tree-form connection of subsystems, i.e., each lower-level system is connected to one and only one upper-level system, which can be achieved by properly splitting the entire system. For two adjacent levels, as shown in Fig. \ref{fig_structure}, variables in each subsystem can be partitioned into two groups,
\begin{figure}[t]
  \centering
  \includegraphics[width=3.3in]{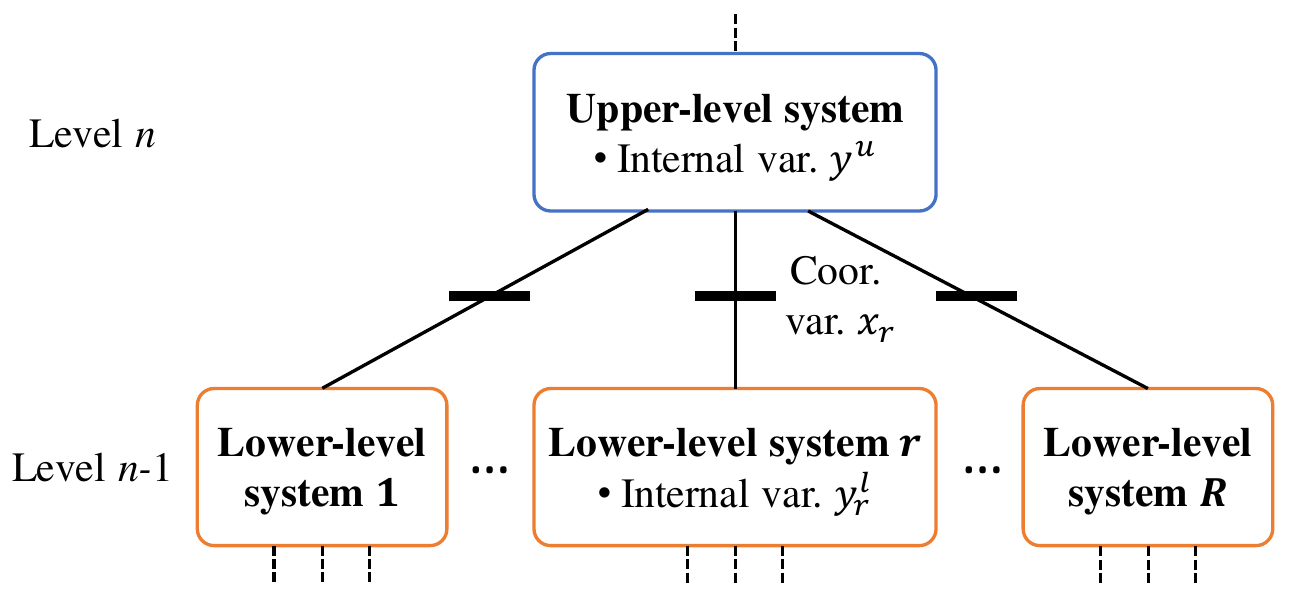}
  \caption{System structure of two adjacent levels.}
  \label{fig_structure}
\end{figure}

\begin{itemize}
  \item Coordination variable $x\in \mathbb{R}^{N_x}$, which contains decision variables at the boundary node, e.g., the power exchange between the two levels;
  \item Internal variable $y\in \mathbb{R}^{N_y}$, which contains internal control variables (e.g., power output of generators) and state variables (e.g., voltage magnitudes and phase angles) of the subsystem. 
\end{itemize}
Following the hierarchical management paradigm, the upper-level system determines the value of the coordination variable, but cannot directly determine the internal variable of each lower-level subsystem. The lower-level subsystem optimally determines the value of the internal variable with the coordination variable as the boundary condition.

In the following contents, we use superscript $u$ and $l$ to label variables of upper-and lower-level systems, respectively. Lower-level systems are indexed by $r \in [R]$, where $R$ is the number of subsystems and $[R]:=\{1,\cdots,R\}$. To concatenate column vectors, we use the syntax $(x_1,\cdots,x_R):=[x_1^\top,\cdots,x_R^\top]^\top$ and $x_{[R]}:=(x_1,\cdots,x_R)$. 

\subsection{Joint Optimal Dispatch}
The joint optimal dispatch (JOD) seeks the most cost-effective generation schedule of the entire multi-level system subject to operation constraints. Without loss of generality, the JOD model is represented as follows,
\begin{subequations}
  \label{jed}
  \begin{align}
    \min_{x,y} \ & C^u(x_{[R]}, y^u) + \sum_{r\in[R]} C^l_r(x_r,y^l_r) \label{jed:obj}\\
    \text{s.t.} \ & h^u(x_{[R]}, y^u) \leq 0,\label{jed:cst_c}\\
    & h_r^l(x_r, y_r^l) \leq 0, \forall r \in [R].\label{jed:cst_i}
  \end{align}
\end{subequations}
Function $C^u(\cdot)$ and $C^l(\cdot)$ respectively denote cost functions of upper-and lower-level systems, which only rely on internal variables. Equations \eqref{jed:cst_c} and \eqref{jed:cst_i} enforce operation constraints of the upper-level system and each lower-level system, respectively.

\subsection{Primal Decomposition}
In the JOD model, the optimization of different subsystems are coupled by coordination variable $x_r$. If the value of $x_r$ is given as $\hat{x}_{r}$, the JOD model can be decomposed into an upper-level problem and the lower-level problem corresponding to each subsystem, 
\begin{itemize}
  \item upper-level problem (UP\_0)
\end{itemize}
\begin{equation}
  \label{model:up_0}
  \min_{y^u} \left\{C^u(\hat{x}_{[R]}, y^u) : h^u(\hat{x}_{[R]}, y^u) \leq 0 \right\},
\end{equation}

\begin{itemize}
  \item lower-level problem (LP\_0)
\end{itemize}
\begin{equation}
  \label{model:lp_0}
  \min_{y^l_r} \left\{C^l(\hat{x}_{r}, y^l_r) : h^l_r(\hat{x}_{r}, y^l_r) \leq 0 \right\}, \forall r\in [R].
\end{equation}

This decomposition scheme is known as the primal decomposition in the optimization literature \cite{lit:dual_5}. The key point of this decomposition is to choose the proper value of $\hat{x}_{r}$ so that the feasibility and optimality of each subsystem are met. In engineering practice, the dispatch for different levels is decoupled and $\hat{x}_{r}$ is determined empirically or based on simplified lower-level system models. For instance, in the transmission network dispatch, distribution networks are simplified as nodal netload. In the pan-European spot market, the cross-regional trading is cleared by simplifying each price area as a virtual node \cite{ref:eu}. However, since the detailed information of lower-level systems is not completely incorporated, $\hat{x}_{r}$ determined based on the simplified model may lead to uneconomical and even infeasible dispatch commands for lower-level systems. To this end, existing coordinated optimization algorithms update $\hat{x}_{r}$ through iterative information exchange among upper-and lower-level systems. In these methods, the upper-level system determines the optimal value of $\hat{x}_{r}$ by solving problem \eqref{model:up_0}. Then each lower-level system checks the feasibility and optimality of $\hat{x}_{r}$ and returns information to the upper-level system to adjust the value of $\hat{x}_{r}$. 

To overcome the drawbacks caused by iterations of conventional coordinated optimization methods, this study seeks the non-iterative solution for coordinating problem \eqref{model:up_0} and \eqref{model:lp_0}. The basic idea is to eliminate internal variables from problem \eqref{model:lp_0} and obtain an equivalent model with reduced scale to replace the original lower-level system model for the upper-level optimization. Detailed introductions of the equivalent system reduction and non-iterative coordinated optimization are included in the following two sections.

\section{Equivalent Projection Theory}\label{sec:theory}
To realize the non-iterative coordinated optimization, this study proposes the EP theory to make external equivalence of the lower-level system model. First, to ensure that internal variables only exist in constraints, the cost function of each lower-level system is transformed into an inequality form using the epigraph. Then, the feasible region characterized by technical and economic constraints of the lower-level system is projected onto the subspace of the coordination variable. Through the projection, internal variables of the lower-level system are eliminated, yielding a reduced model to equivalently represent the lower-level system model with much fewer data. 

\subsection{Reformulation of Objective Function}\label{chp:epigraph}
In the original JOD model \eqref{jed}, internal variable $y^l_r$ of the lower-level system not only appears in operation constraint \eqref{jed:cst_i}, but also exists in the objective function, which impedes the decomposition of the JOD model. Hence, we convert the lower-level objective function into the inequality-form via the epigraph. With this conversion, the JOD model is reformulated as follows,
\begin{subequations}
  \label{jed2}
  \begin{align}
    \min_{x,\pi,y} \ & C^u(x_{[R]}, y^u) + \sum_{r\in[R]} \pi_r \label{jed2:obj}\\
    \text{s.t.} \ & h^u(x_{[R]}, y^u) \leq 0,\label{jed2:cst_c}\\
    & h_r^l(x_r, y_r^l) \leq 0, \forall r \in [R],\label{jed2:cst_i}\\
    & C^l_r(x_r, y_r^l) \leq \pi_r \leq \overline{\pi}_r, \forall r \in [R].\label{jed2:cst_e}
  \end{align}
\end{subequations}
Equation \eqref{jed2:cst_e} represents the epigraph of the lower-level cost function $C^l_r(x_r, y_r^l)$. Variable $\pi_r$ denotes the operation cost of lower-level system $r$. Constant $\overline{\pi}_r$ is introduced to bound the value of $\pi_r$, which can take any value larger than the supremum of $C^l_r(x_r, y_r^l)$. 

\begin{theorem}
  \label{theorem:jed}
  The reformulated JOD model \eqref{jed2} is equivalent to the original JOD model \eqref{jed}.
\end{theorem}

\begin{proof}
  Note that problem \eqref{jed2} minimizes over $\pi_r$, the first less-than-equal sign in equation \eqref{jed2:cst_e} will always get binding at the optimum, i.e., $C^l_r(\hat{x}_r, \hat{y}^{L}_r) = \hat{\pi}_r$ always holds for optimal solution $(\hat{x}_r, \hat{y}^{L}_r, \hat{\pi}_r)$. This is because if the constraint is not binding, a value of $\pi_r$ smaller than $\hat{\pi}_r$ can be found to further decrease the value of the objective function \eqref{jed2:obj}, which conflicts with that $\hat{\pi}_r$ is the optimal solution. Hence, the reformulated JOD model \eqref{jed2} is equivalent to the original model \eqref{jed}.
\end{proof}

In the reformulated JOD model, regard $(x_r, \pi_r)$ as the augmented coordination variable of subsystem $r$. Then internal variable $y^l_r$ only appears in the lower-level operation constraint, and the upper-and lower-level systems are only coupled through $(x_r, \pi_r)$. 

\subsection{Equivalent Projection}
The target of the EP is to eliminate internal variables from the optimization model of each lower-level system to obtain its external equivalence. First, define the operation feasible region (OFR) of each lower-level system,
\begin{definition}[OFR]\label{def:ofr}
  The OFR of lower-level system $r$ is the feasible region of both coordination variables and internal variables subject to the operation constraints, i.e., $\Omega_r := \left\{(x_r, \pi_r, y^l_r)\in \mathbb{R}^{N_x+1} \times \mathbb{R}^{N_y} : \text{Eq } \eqref{jed2:cst_i}-\eqref{jed2:cst_e} \right\}$.
\end{definition}

Then define the EP of the lower-level system,
\begin{definition}[EP]\label{def:ESR}
  The EP model of lower-level system $r$ is the projection of $\Omega_r$ onto the subspace of $(x_r,\pi_r)$, i.e., $\Phi_r := \left\{(x_r,\pi_r) \in \mathbb{R}^{N_x+1} : \exists y^l_r, \ \text{s.t. } (x_r, \pi_r, y^l_r) \in \Omega_r \right\}$.
\end{definition}

For any coordination variable $(x_r,\pi_r)$ satisfies the EP model, the above definition ensures that there is at least one feasible operation state $y^l_r$ to execute the dispatch command $x_r$ with operation cost no larger than $\pi_r$. Through the projection, internal variables are eliminated from the lower-level optimization model, yielding a lower-dimensional feasible region to depict the technical and economic features of the system. The resulted EP model contains all possible values of coordination variables that can be executed by the lower-level system securely and economically and thus, can be used as the substitute for the detailed lower-level model in the coordinated optimization.

\subsection{Discussions}
\subsubsection{Privacy protection}
In the coordinated dispatch, private and sensitive information of lower-level systems should not be revealed, which typically refers to network parameters, nodal loads, and cost functions of resources \cite{ref:privacy1, ref:privacy2}. The privacy protection of the EP is guaranteed by the following theorem.
\begin{theorem}
  The original model of the lower-level system cannot be inferred from the explicit representation of its EP model.
\end{theorem}
\begin{proof}
  From Definition \ref{def:ESR}, it is obvious that $\Phi_r$ is the projection of $\Phi_r$ itself. $\Phi_r$ is also the projection of set $\left\{(x_r,\pi_r, 0), (x_r,\pi_r, 1) \right\}$ where $(x_r,\pi_r) \in \Phi_r$. Hence, one EP model may be the projection of different OFRs and thus, the original system model cannot be referred from only the EP model. 
\end{proof}

The projection is a many-to-one mapping from the high-dimensional region to the low-dimensional space, which endows the privacy-protection property of the EP. One may concern if the lower-level system model can be recovered from a series of EP models based on historical data. Actually, this is neither unachievable since the dimension of internal variable $y^l_r$ cannot be referred from EP models, let alone the detailed model of the lower-level system.

\subsubsection{Comparison with network reduction}
The EP is reminiscent of the power network reduction such as the Ward equivalence. The network reduction technique eliminates internal variables from the network equation and obtains a reduced equation for only boundary nodes. The network reduction is mathematically formulated as a matrix inversion problem and can be computationally solved by the Gaussian Elimination. The network reduction makes equivalence of the network equation at the boundary nodes. However, operation limits and costs of components in the network are neglected during the reduction. Hence, the network reduction can only be used in network analysis (e.g., power flow calculation, state estimation), but cannot be used for the coordinated optimization. In contrast, the proposed EP method eliminates internal variables from the entire optimization model containing not only the network equation, but also inequality constraints and operation cost. The resulted EP model extracts both technical and economic features of the subsystem to its boundary. Therefore, the EP model can be used to replace the optimization model of the lower-level system in the upper-level coordinated optimization. In this regard, the proposed EP method extends the equivalence theory of power systems.

\subsubsection{Comparison with feasible region projection}
In the recent literature, the feasible region projection is employed for operation constraints reduction. For the multi-area transmission system, the projected region is used to characterize the allowable range of tie-line power that can be executed by the regional system \cite{ref:tan1}. For the active distribution network, the projected region is used to characterize the admissible range of flexibility the distribution system can provide \cite{ref:tan2, ref:tan3, ref:minkowski, ref:appr}. Though the technical feasibility of the lower-level system can be guaranteed via the feasible region projection, the economic characteristics are not finely incorporated and thus, the optimality of the coordination is not guaranteed. With the EP in the present paper, in contrast, economic and technical characteristics of the lower-level system are projected onto the coordination space simultaneously. As will be proven in the following section, using the EP model to replace the original lower-level system model in the coordinated optimization, the coordinated dispatch result will be consistent with the feasibility and optimality of the lower-level system. 


\section{Coordinated Optimization Based on EP}
\subsection{EP-based Model Decomposition}
In the proposed framework, each lower-level system submits its EP model as a substitute for its original model to participate in the upper-level coordinated dispatch. Thereby, the JOD model is equivalently decomposed into the following two levels,
\begin{itemize}
  \item upper-level problem (UP\_1)
\end{itemize}
\begin{subequations}
  \label{model:up_1}
  \begin{align}
    \min_{x,\pi,y^u} \ & C^u(x_{[R]}, y^u) + \sum_{r\in[R]} \pi_r\\
    \text{s.t.} \ & h^u(x_{[R]},y^u) \leq 0,\\
    & (x_r, \pi_r) \in \Phi_r, \forall r\in [R].
  \end{align}
\end{subequations}

\begin{itemize}
  \item lower-level problem (LP\_1)
\end{itemize}
\begin{subequations}
  \label{model:lp_1}
  \begin{align}
    \min_{y^l_r} \ & C^l_r(\hat{x}_r, y^l_r)\\
    \text{s.t.} \ & h^l_r(\hat{x}_r,y^l_r) \leq 0,\\
    & C^l_r(\hat{x}_r, y^l_r) \leq \hat{\pi}^l_r.\label{model:lp_1_cost}
  \end{align}
\end{subequations}
In problem LP\_1, $\hat{x}_r$ and $\hat{\pi}^l_r$ are optimal solution of the upper-level problem UP\_1. Each lower-level system dispatches its local system with the upper-level decision result as the boundary condition.

\subsection{Non-Iterative Coordination Scheme}\label{sec:scheme}
Based on the EP, the optimal dispatch of upper-and lower-level systems can be coordinated in a non-iterative fashion. The proposed coordination scheme contains three stages, i.e., system reduction, coordinated optimization, and subsystem operation. The information exchange procedure of the proposed coordination scheme is illustrated in Fig. \ref{fig_process}. Detailed introductions are as follows.
\begin{figure}[t]
  \centering
  \includegraphics[width = 3.5 in]{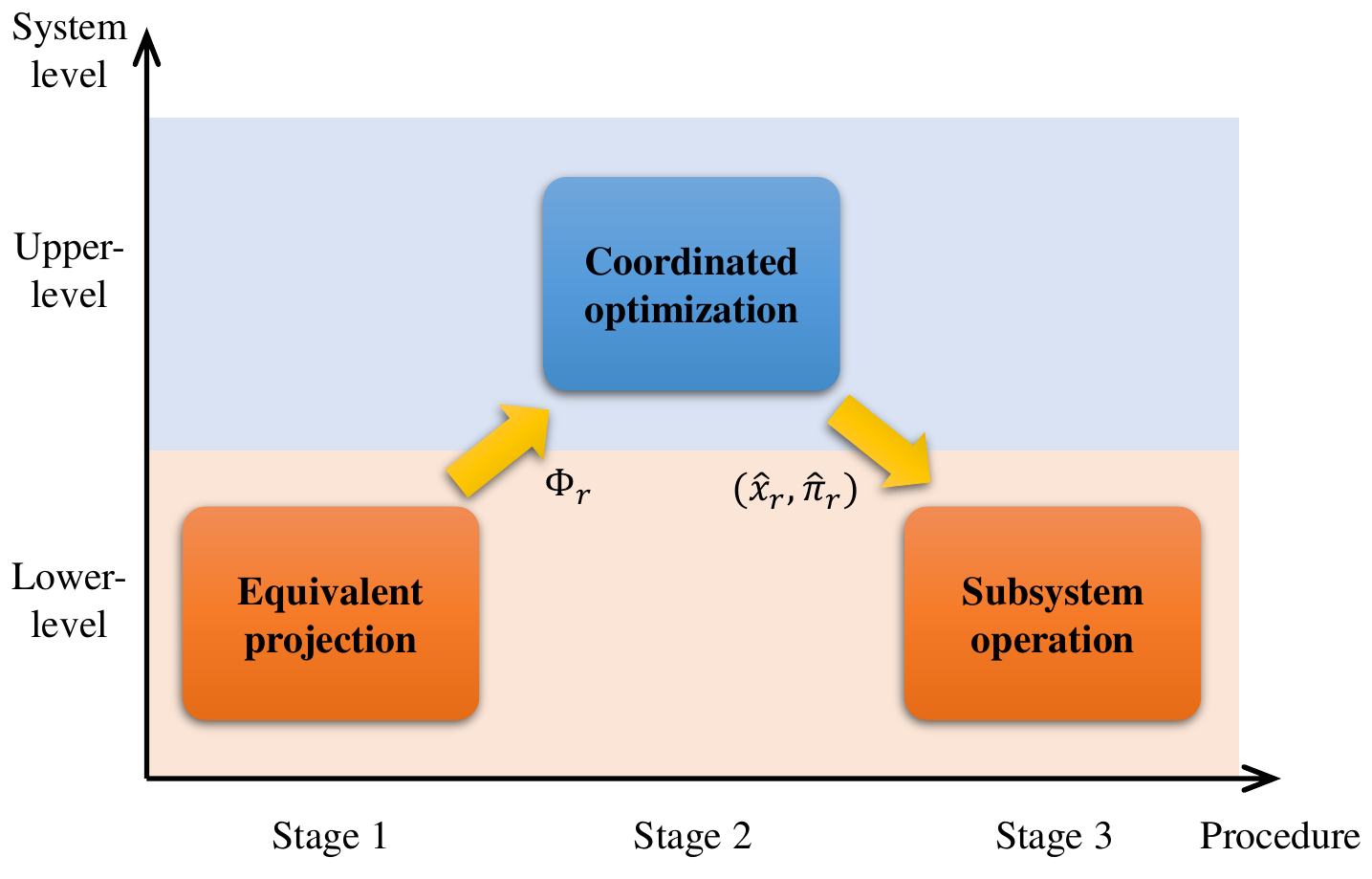}
  \caption{Procedure of the EP-based coordinated optimization.}
  \label{fig_process}
\end{figure}

\begin{itemize}
  \item \textbf{Stage 1: equivalent projection.} Each lower-level system calculates its EP model $\Phi_r$ according to Definition \ref{def:ESR} and submits $\Phi_r$ to the upper-level system. 
  \item \textbf{Stage 2: coordinated optimization.} The upper-level system solves problem \eqref{model:up_1} with $\Phi_r$ of each lower-level system as constraints. The optimal solution for the coordination variable is $(\hat{x}_r,\hat{\pi}_r)$, which is published to lower-level systems as dispatch command. 
  \item \textbf{Stage 3: subsystem operation.} Each lower-level system fixes the coordination variable to $(\hat{x}_r,\hat{\pi}_r)$ and solves problem \eqref{model:lp_1} to dispatch the local system. Since $(\hat{x}_r,\hat{\pi}_r) \in \Phi_r$, problem \eqref{model:lp_1} is ensured to be feasible with $(\hat{x}_r,\hat{\pi}_r)$ as the boundary condition.
\end{itemize}

The above coordination scheme only requires one round of interaction between upper-and lower-level systems, which overcomes drawbacks of conventional iterative coordination algorithms such as slow convergence, risk of iteration oscillation, and complicated information exchange. Additionally, the proposed coordination procedure is serial, which is compatible with the existing hierarchical management paradigm of power systems. 

\subsection{Discussions}
\subsubsection{Optimality} 
The optimality of the EP-based coordination scheme is proven as follows.
\begin{theorem}
  The coordinated solution of problem \eqref{model:up_1} and problem \eqref{model:lp_1} is equivalent to that of the JOD.
\end{theorem}
\begin{proof}
  In problem \eqref{model:up_1}, its optimal solution $(\hat{x}_r,\hat{\pi}_r) \in \Phi_r$. According to Definition \ref{def:ESR}, there is $\hat{y}^l_r$ such that $(\hat{x}_r,\hat{\pi}_r, \hat{y}^l_r)$ satisfies constraint \eqref{jed2:cst_i} and \eqref{jed2:cst_e} and thus, $(\hat{x}_r,\hat{\pi}_r, \hat{y}^u,\hat{y}^l_r)$ is feasible for problem \eqref{jed2}. Let $(\breve{x}_r,\breve{\pi}_r, \breve{y}^u,\breve{y}^l_r)$ be the optimal solution of problem \eqref{jed2}. According to Definition \ref{def:ESR}, $(\breve{x}_r,\breve{\pi}_r) \in \Phi_r$ and thus, $(\breve{x}_r,\breve{\pi}_r, \breve{y}^u)$ is feasible for problem \eqref{model:up_1}. Hence, problem \eqref{model:up_1} and \eqref{jed2} have the same feasible region. Note that problem \eqref{model:up_1} and \eqref{jed2} also have identical objective functions, the optimal solution of problem \eqref{model:up_1} is equivalent to that of problem \eqref{jed2}. According to Theorem \ref{theorem:jed}, problem \eqref{jed2} is equivalent to the original JOD problem \eqref{jed} and thus, the solution of problem \eqref{model:up_1} is equivalent to the JOD. With the minimized cost $\hat{\pi}_r$, constraint \eqref{model:lp_1_cost} will take the equal sign at the optimum of problem \eqref{model:lp_1}, otherwise a $\pi^l_r$ smaller than $\hat{\pi}_r$ can be found, which conflicts with that $\hat{\pi}_r$ is the optimal value. Hence, the optimal solution of problem \eqref{model:lp_1} is equal to that of \eqref{jed}.
\end{proof}

Note that the above theorem and proof do not restrict the form of the JOD model. Hence, the EP-based coordinated optimization framework is general, and is capable of both convex and nonconvex problems. In addition to the mathematical proof, the optimality of the EP-based coordination can also be analyzed through physical interpretation. According to Definition \ref{def:ESR}, the EP model contains all coordination variables that are technically and economically feasible for the lower-level system. Hence, the upper-level decision constrained by the EP model can be executed by the lower-level system with minimized operation cost, which ensures the equivalence between the EP-based coordination and the original JOD solution.

\subsubsection{Complexity analysis}
Let $G^u$ and $G^l_r$ denote the model scale of the upper-level and lower-level problems, respectively. The problem scale can be measured by the product of the variable number and constraint number of the problem (assume there is no redundant variable or constraint). Given the form of the optimization problem, the computation time of the problem is an increasing function of its scale, denoted by $\mathcal{T}^o(\cdot)$. For the lower-level system, the time of system reduction is an increasing function of the problem scale and the coordination variable dimension, denoted by $\mathcal{T}^p(G^l_r,N_x)$. In the EP-based coordination scheme, the total computation time is the sum of the 3 stages. Note that both Stage 1 and Stage 3 are implemented simultaneously by different lower-level systems and thus, the computation time of the corresponding stage relies on the longest time of lower-level systems. Hence, the total computation time of the EP-based coordination scheme is
\begin{equation}
  \mathcal{T}^{coor} = \max_{r\in[R]} \mathcal{T}^p(G^l_r,N_x) + \mathcal{T}^o(G^u) + \max_{r\in[R]}\mathcal{T}^o(G^l_r).
\end{equation}

The total computation time of directly solving the JOD is
\begin{equation}
  \mathcal{T}^{jod} = \mathcal{T}^o(G^u + \sum_{r\in[R]} G^l_r).
\end{equation}

If the scale of each system is given, the computation time of solving the JOD will increase with the number of lower-level systems. In contrast, the computation time of the EP-based coordination method mainly depends on the longest time of lower-level systems and will not be significantly impacted by the number of lower-level systems. The reason is that in the EP-based coordination scheme, each lower-level system reduces its original model into the small-scale EP model to participate in the upper-level optimization, which alleviates the computation burden of the upper-level system. Though more computation efforts are required by Stage 1 and Stage 3, they are executed by each lower-level system in parallel and do not add much to the overall computation time. In this regard, the EP-based coordination method distributes the computation complexity of the joint optimization to subsystems and will improve computation efficiency, especially for coordinated optimization problems with numerous lower-level systems. This advantage is verified based on detailed application scenarios in Part II of this paper. The projection algorithm is a key factor that impacts the computation time of the EP-based coordination scheme, which is also addressed in Part II. 

\subsubsection{Application to multi-level system}
The EP-based framework is also capable of coordinated optimization for the hierarchical system with multiple levels. The coordination for the multi-level system also requires three stages similar to the process in Section \ref{sec:scheme}. As illustrated in Fig. \ref{fig_hierarchical}, Stage 1 is implemented in a `bottom-up' manner. The system reduction calculation starts from the lowest level and then, systems at higher levels calculate their EP models with lower-level EP models as constraints. This process continues until it comes to the highest level, which optimizes the value of coordinated variables incorporating EP models of its connected lower-level systems. Then Stage 2 is implemented in a `top-down' manner to disaggregate the dispatch command. Systems at each level solve their local optimal dispatch problems with dispatch command from the upper-level system as the boundary condition. This process continues until it comes to the lowest level. As can be seen, the EP-based coordination process does not require multiple iterations for the multi-level system. 
\begin{figure}[t]
  \centering
  \includegraphics[width = 3.1 in]{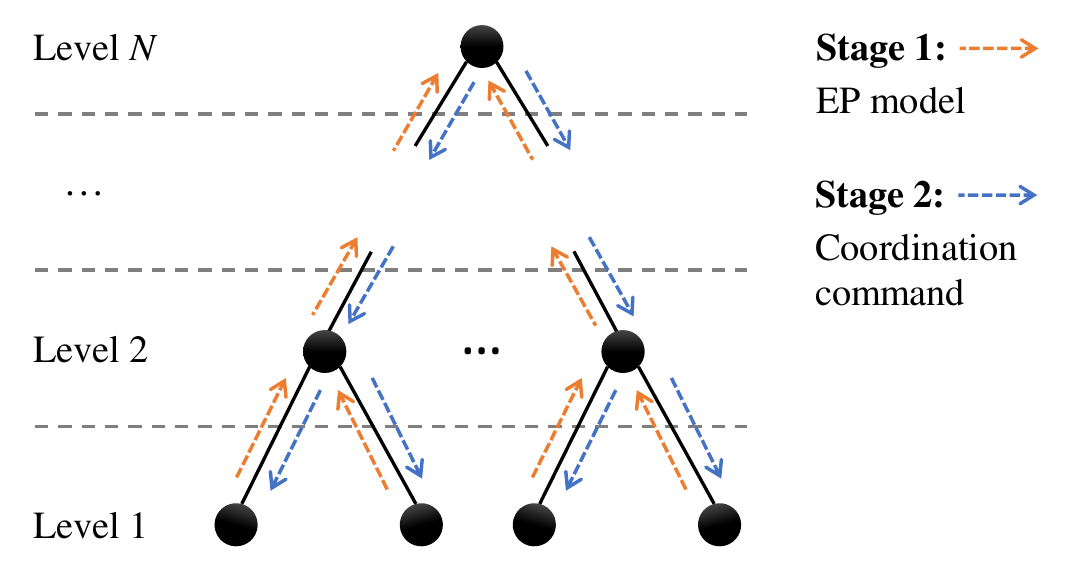}
  \caption{Coordination scheme for the multi-level system.}
  \label{fig_hierarchical}
\end{figure}

If iterative coordinated optimization methods are applied to the multi-level system, the total number of iterations among different levels will be $\mathcal{O}(m^n)$, where $m$ is the number of iterations required by the coordination between two adjacent levels and $n$ is the number of levels. The iteration number of iterative coordination methods is exponential to the number of system levels, making it difficult to be applied to multi-level systems. This drawback is overcome naturally by the EP-based method since no iteration is needed no matter how many levels the system has. 

\begin{figure*}[t]
  \centering
  \includegraphics[width = 7.16 in]{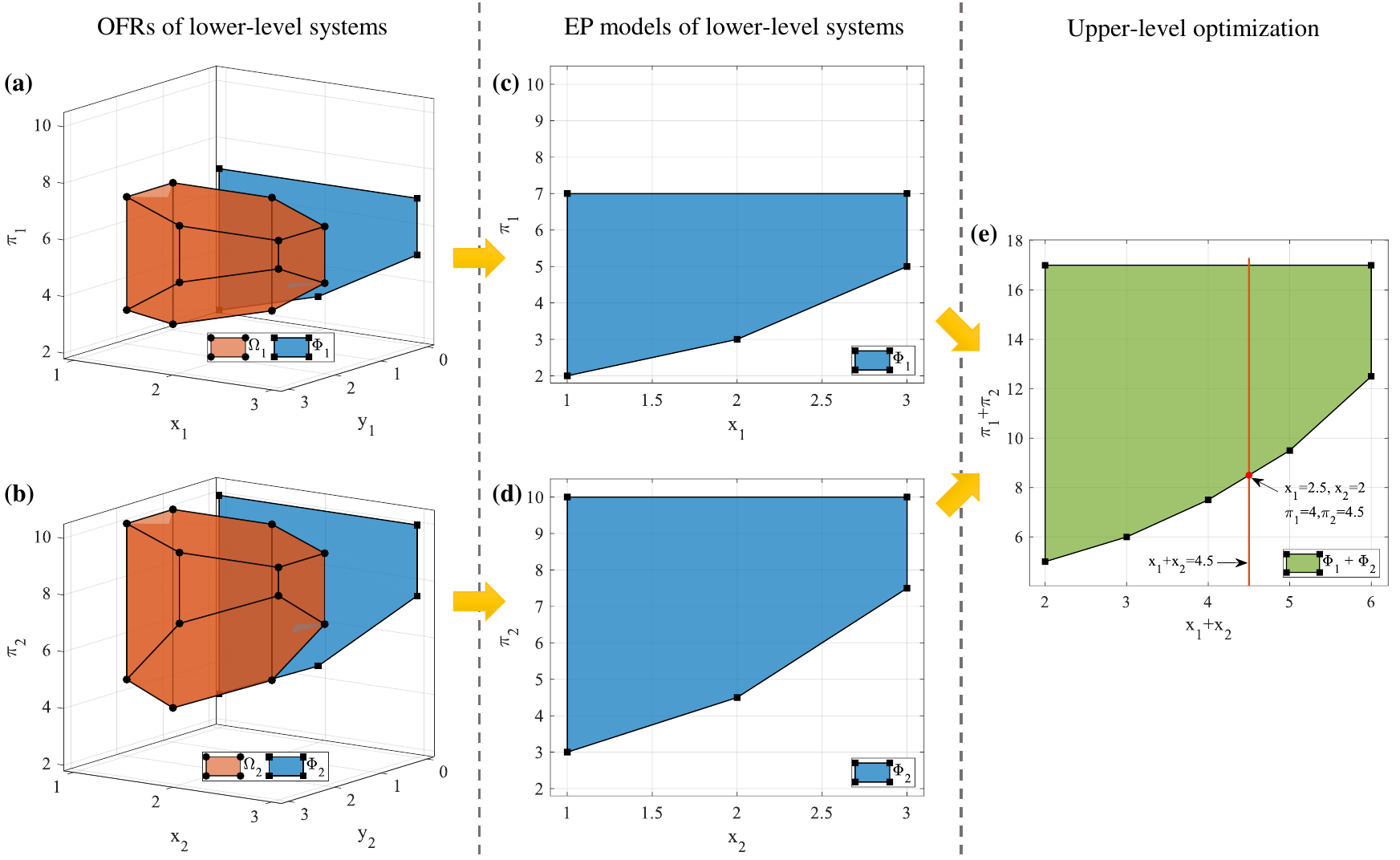}
  \caption{Illustrative example results. (a) and (b) are OFRs of lower-level system 1 and system 2. (c) and (d) are EP models of lower-level systems in the coordination space. (e) exhibits the upper-level optimization results.}
  \label{fig_illu}
\end{figure*}

\section{Illustrative Example}
The implementation of the EP-based coordinated optimization is exhibited via an illustrative example. The test system is composed of an upper-level system and two lower-level systems. The JOD of the test system takes the form of \eqref{jed} and is parametrized as follows,
\begin{subequations}
  \label{case_md}
  \begin{align}
    \min_{x,\pi,y} \ & \ \pi_1 + \pi_2  \label{case_md:obj}\\
    \text{s.t.} \ & \ x_1 + x_2 = 4.5 \rightarrow \text{Upper-level Constraint} \label{case_md:up}\\
    & \left. \begin{array}{l}
      1 \leq x_1 \leq 3,\  1 \leq y_1 \leq 3\\
      -1 \leq -x_1+y_1 \leq 1\\
      x_1 + y_1 \leq \pi_1 \leq 7\\
    \end{array}
    \right\} \text{Lower Sys. 1} \label{case_md:low1}\\
    & \left. \begin{array}{l}
      1 \leq x_2 \leq 3,\  1 \leq y_2 \leq 3\\
      -1 \leq -x_2+y_2 \leq 1\\
      1.5 \times (x_2 + y_2) \leq \pi_2 \leq 10\\
    \end{array}
    \right\} \text{Lower Sys. 2} \label{case_md:low2}
  \end{align}
\end{subequations}
In the above model, $x_1$ and $x_2$ are coordination variables of the two lower-level systems, respectively. $y_1$ and $y_2$ are internal variables. $\pi_1$ and $\pi_2$ are cost variables. The objective of the JOD in \eqref{case_md:obj} is to minimize the total cost of the lower-level systems. Equation \eqref{case_md:up} is the constraint of the upper-level system, which couples the optimization of the lower-level systems. Equation \eqref{case_md:low1} and \eqref{case_md:low2} are constraints of the lower-level systems. The JOD problem \eqref{case_md} is a linear programming. By directly solving the problem, its optimal solution is
\begin{equation}
  \label{case:rst1}
  (\hat{x}_1,\hat{y}_1,\hat{\pi}_1) = (2.5, 1.5, 4)
\end{equation}
\begin{equation}
  \label{case:rst2}
  (\hat{x}_2,\hat{y}_2,\hat{\pi}_2) = (2, 1, 4.5)
\end{equation}

According to Definition \ref{def:ofr}, operation feasible regions of the two lower-level systems are as follows
\begin{equation}
  \label{case:ofr1}
  \Omega_1 = \left\{(x_1, \pi_1, y_1) \in \mathbb{R}^3: \text{Eq. } \eqref{case_md:low1} \right\},
\end{equation}
\begin{equation}
  \label{case:ofr2}
  \Omega_2 = \left\{(x_2, \pi_2, y_2) \in \mathbb{R}^3: \text{Eq. } \eqref{case_md:low2} \right\}.
\end{equation}
Both $\Omega_1$ and $\Omega_2$ are 3-dimensional polytopes, as red regions show in Fig. \ref{fig_illu} (a) and (b). 

The problem scale of this test case is small and the projection can be calculated by the classic Fourier-Motzkin Elimination (FME) method. Project $\Omega_1$ and $\Omega_2$ onto the subspace of $(x,\pi)$, EP models of lower-level systems are obtained, as blue regions illustrate in Fig. \ref{fig_illu} (c) and (d). In this case, EP models are 2-dimensional polygons, which are represented as follows,
\begin{equation}
  \label{case:ESR1}
  \begin{split}
    \Phi_1 = \{(x_1,\pi_1) \in \mathbb{R}^2:1 \leq x_1 \leq 3, \pi_1 \leq 7,\\    
    x_1 - \pi_1 \leq -1, 2 x_1 - \pi_1 \leq 1\},
  \end{split}
\end{equation}

\begin{equation}
  \label{case:ESR2}
  \begin{split}
    \Phi_2 = \{(x_2,\pi_2) \in \mathbb{R}^2:1 \leq x_2 \leq 3, \pi_2 \leq 10,\\    
    3 x_2 - 2 \pi_2 \leq -3, 6 x_2 - 2 \pi_2 \leq 3 \}.
  \end{split}
\end{equation}

The two lower-level systems submit their EP models as substitutes for their original models to form the upper-level coordinated optimization problem. In this illustrative case, the upper-level optimization problem can be solved by the graphical method. The objective in \eqref{case_md:obj} is to minimize the sum of $\pi_1$ and $\pi_2$, and the upper-level constraint is the limitation on the sum of $x_1$ and $x_2$. Take $x_1+x_2$ and $\pi_1+\pi_2$ as decision variables, then the feasible region of the upper-level problem is the intersection of the Minkowski sum of EP models of lower-level systems and the line $x_1+x_2=4.5$, as illustrated in Fig. \ref{fig_illu} (e). From the figure, it can be obtained that the optimal solution of the upper-level problem is: $\hat{x}_1 = 2.5$, $\hat{x}_2 = 2$, $\hat{\pi}_1 = 4$, $\hat{\pi}_2 = 4.5$. Fix the upper-level decision results in \eqref{case_md:low1} and \eqref{case_md:low2}, also note that $\hat{\pi}_1$ and $\hat{\pi}_2$ are minimized, then it can be inferred that $\hat{y}_1 = 1.5$ and $\hat{y}_2 = 1$. As can be seen, the solution based on the EP-based coordinated optimization is equivalent to the joint optimization solution in \eqref{case:rst1} and \eqref{case:rst2}, which validates the effectiveness of the proposed coordination method.

\section{Conclusion}
To overcome drawbacks brought by repetitive iterations of conventional coordinated optimization methods, this paper proposes a novel reduction theory namely the EP, and develops a non-iterative COD framework for hierarchical power systems. The EP eliminates internal variables from technical and economic constraints of the lower-level system, and makes external equivalence of the entire optimization model of the system with much fewer data. In the proposed COD framework, the EP model is used to replace the original lower-level system model in coordinated optimization. With the EP, only a single-round exchange of some boundary information of lower-level systems is required to achieve the coordinated optimality, which avoids iterations among subsystems. The EP-based COD method is proven to protect private information, guarantee the same optimality as the joint optimization, and is capable of multi-level coordinated optimization problems. A numerical example demonstrates the detailed process of the EP-based coordination framework. 

In Part II of this paper, the methodology for calculating the EP model will be introduced. The performance of the EP-based non-iterative COD will also be tested based on specific applications.



\ifCLASSOPTIONcaptionsoff
  \newpage
\fi


\bibliographystyle{IEEEtran}
\bibliography{Reference}

\end{document}